\documentclass[11pt]{article}

\usepackage{graphicx, amssymb, latexsym, amsfonts, amsmath, lscape, amscd,
amsthm, color, epsfig, mathrsfs, tikz, enumerate}

\newtheorem{theorem}{Theorem}[section]
\newtheorem{conjecture}[theorem]{Conjecture}
\newtheorem{corollary}[theorem]{Corollary}

\newtheorem{lemma}[theorem]{Lemma}

\newtheorem{problem}[theorem]{Problem}

\newcommand\DELETE[1]{}

\begin{document}

\title{ Walk-powers and homomorphism bound of planar graphs}%
\author{ 
{\sc Reza Naserasr}$\, ^a$, {\sc Sagnik Sen}$\, ^b$, 
{\sc Qiang Sun\thanks{Corresponding author, supported by CSC. E-mail address: sun@lri.fr(Q. Sun).}}$\, ^{~a}$\\
\mbox{}\\
{\small $(a)$ LRI, CNRS and Universit\'e Paris Sud, F-91405 Orsay Cedex, France}\\
{\small $(b)$ Indian Statistical Institute, Kolkata, India}
}
%
 
\date{\today}
\maketitle

\begin{abstract}
As an extension of the Four-Color Theorem it is conjectured that every planar 
graph of odd-girth at least $2k+1$ admits a homomorphism to 
$PC_{2k}=(\mathbb{Z}_2^{2k}, \{e_1, e_2, \cdots,e_{2k}, J\})$
where $e_i$'s are standard basis and $J$ is all 1 vector. Noting that $PC_{2k}$ itself
is of odd-girth $2k+1$, in this work we show that if the conjecture is true, then
$PC_{2k}$ is an optimal such a graph both with respect to number of vertices
and number of edges. 
The result is obtained using the notion of walk-power of graphs and their
clique numbers. 

An analogous result is proved for bipartite signed planar graphs of unbalanced-girth 
$2k$. The work is presented on a uniform frame work of planar consistent signed graphs.

\end{abstract}



\section{Introudction}

\subsection{Signed graphs, notation}
Given a graph $G$, a \emph{signature} on $G$ is an assignment of signs, $+$
or $-$, to the edges. The set of negative edges is normally denoted by $\Sigma$ 
and will normally be referred to as the signature. A \emph{re-signing} is to change
the signs of all edges incident to a given set of vertices or, equivalently,
edges of an edge-cut. Two signatures are said to be equivalent if one
can be obtained from the other by a re-signing. A graph $G$ together with a set of signatures equivalent to 
$\Sigma$ is called a signed graph and is denoted by $[G, \Sigma]$ where $\Sigma$ is any member of the class
of equivalent signatures. A signed cycle with an even (odd) number 
of negative edges is called \emph{balanced} (\emph{unbalanced}). It is easily observed that the balance
of a cycle is invariant of re-signing. The \emph{unbalanced-girth} of a signed graph is the shortest
length of its unbalanced cycles. A \emph{consistent} signed graph is a signed graph in which every balanced
cycle is of even length and all unbalanced cycles are of a same parity. Thus there are two types of consistent signed
graphs: 
\begin{itemize}
 \item[i.] when all unbalanced cycles are of odd length, it can be shown that this is 
 the case if and only if 
$\Sigma \equiv E(G)$, such a signed graph will be called an \emph{odd signed graph}; 
\item[ii.] 
 when all unbalanced cycles are of even length, which will be the case if and only if $G$ is bipartite,
such a signed graph thus will be referred to as a \emph{singed bipartite graph}.
\end{itemize}

\subsection{Homomorphisms and bounds}
Given two signed graphs $[G, \Sigma]$ and $[H, \Sigma']$ we say there is 
a \emph{homomorphism} of $[G, \Sigma]$
to $[H, \Sigma']$, denoted $[G, \Sigma] \to [H, \Sigma']$, if there is 
a signature $\Sigma_1$ of $G$ equivalent to
$\Sigma$ and a mapping $\varphi: V(G)\to V(H)$ such that $xy \in E(G)$ 
implies $\varphi(x)\varphi(y)\in E(H)$ and 
$xy\in \Sigma_1$ if and only if $\varphi(x)\varphi(y)\in \Sigma'$. 
It is easy to prove that if 
$[G, \Sigma] \to [H, \Sigma']$, then unbalanced-girth of $[G, \Sigma]$ is 
at least as the unbalanced-girth of $[H, \Sigma']$.
Give a class $\mathcal C$ of (signed) graphs we say a (signed) graph $B$ 
\emph{bounds} $\mathcal{C}$ if every member of
$\mathcal{C}$ admits a homomorphism to $B$. For more on this subject we refer
to \cite{NRS14}.

\subsection{Signed Projective cubes}
\emph{Projective cube} of dimension $d$, denoted $\mathcal{PC}_d$, is the Cayley graph 
$(\mathbb{Z}_2^d, \{ e_1, e_2, \cdots e_d, J \})$ where $e_i$'s are the standard basis and $J$ is the all 1 vector
of relevant length ($d$ here). It is obtained by identifying antipodal vertices of the hypercube of dimension $d+1$ 
or, equivalently, by adding edges between pairs of antipodal vertices of hypercube of dimension $d$.
We define \emph{singed projective cube} of dimension $d$, denoted $\mathcal{SPC}_d$,
to be the singed graph obtained from $\mathcal{PC}_d$ by assigning $+$ to each edge corresponding to an $e_i$ and
$-$ to edges corresponding to $J$.

Projective cubes, also known as folded cubes, are well-studied graphs. We refer to \cite{NRS14} and references there for
some properties of signed projective cubes and for a proof of the following two theorems:

\begin{theorem}
 Signed projective cube of dimension $d$ is a consistent signed graph and has unbalanced-girth $d+1$. 
\end{theorem}

It follows that if a signed graph admits a homomorphism to a signed projective cube, it must be a consistent signed graph.
Such a mapping then becomes equivalent to a packing problem as the following theorem claims:

\begin{theorem}
 A consistent signed graph $(G,\Sigma)$ admits a homomorphism 
 to $\mathcal{SPC}_d$ if and only if the edges set of $G$ 
 can be partitioned into $d+1$ disjoint sets each of which induces
 a signature equivalent to $\Sigma$.
\end{theorem}

The following conjecture, introduced in \cite{N07} and \cite{G05} 
(also see \cite{NRS13}) is the focus of this work:

\begin{conjecture}\label{PlanarsToProjectiveCubes}
 Given $d\geq 2$, every planar consistent signed graph of unbalanced-girth 
 $d+1$ admits a homomorphism to $\mathcal{SPC}_d$.
\end{conjecture}

The conjecture is formed of two parts: for even values of $d$ 
(by considering the signature in which all edges are negative) it 
claims that every planar graph of odd-girth
at least $d+1$ admits a homomorphism to $\mathcal{PC}_d$. 
For odd values of $d$ it says that every planar signed bipartite
graph of unbalanced-girth at least $d+1$ admits a homomorphism to 
$\mathcal{SPC}_d$. Since $\mathcal{PC}_2$ is isomorphic to $K_4$,
the very first case of this conjecture is the Four-Color Theorem.

\subsection{Question of Ne\v set\v ril}
This conjecture, for odd values of $d$ was introduced in \cite{N07} in relation 
to a question of J.~Ne\v set\v ril who asked if there is a triangle-free graph 
to which every triangle-free planar graph admits a homomorphism. This question 
was answered in a larger frame work by P. Ossona de Mendez and 
J.~Ne\v set\v ril which is stated based on the following notation.

Given a finite set $\mathcal H$ of connected graphs we use $Forb_h(\mathcal
H)$ to denote the class of all graphs which do not admit a homomorphism from
any member of $\mathcal H$. Similarly, given a set $\mathcal M$ of graphs we
use $Forb_m(\mathcal M)$ to denote the class of all graphs that have no member
of $\mathcal M$ as a minor.

\begin{theorem}\cite{NO08}\label{NOduality}
Given a finite set $\mathcal M$ of graphs and a finite set $\mathcal H$ of 
connected graphs, there is graph in $Forb_h(\mathcal H)$ to which every graph 
in $Forb_h(\mathcal H) \cap Forb_m(\mathcal M)$ admits a homomorphism.
\end{theorem}

\subsection{In this paper}
The bound that are build using known proofs of this theorem are super exponential.
To find the optimal bound in this theorem, in general, is a very difficult question. 
Indeed this question, in particular, contains the Hadwiger's conjecture simply by 
taking $\mathcal M=\mathcal H=\{K_n\}$. Conjecture~\ref{PlanarsToProjectiveCubes}
proposes a smaller bound for the case of $M=\{K_5, K_{3,3}\}$ and $H=\{C_{2k-1}\}$. 
For $k=1$, ($C_1$ being a loop), since $K_4$ is a planar graph, it is the optimal 
answer by the Four-Color Theorem. For $k=2$, it is proved in \cite{N13} that $PC(4)$,
known as the Clebsch graph, is the optimal bound. Here we prove that any bound of 
odd-girth $2k+1$ for planar graphs of odd-girth $2k+1$ has to have at least $2^{2k}$ 
vertcies each of degree at least $2k+1$. This would imply that if 
Conjecture~\ref{PlanarsToProjectiveCubes} holds, then  $PC(2k)$ is an optimal bound.
We prove an analogue result for the case of planar signed bipartite graphs, 
even though analogue of Theorem~\ref{NOduality} for signed bipartite graphs is not 
proved yet.

\section{Optimal bound for planar odd signed graphs}

In this section we consider the first part of 
Conjecture~\ref{PlanarsToProjectiveCubes}. This case deals with odd signed graphs in
which case one can assume all the edges are negative. 
Thus homomorphism problem here is simply a homomorphism of graphs.

To prove our result, in fact we prove a stronger claim in the following sense. 
Given a graph $G$ and a positive integer $k$
we define the $k$-th walk-power of $G$, denoted by $G^{(k)}$, to be a graph whose vertex set is also $V(G)$ with two 
vertcies $x$ and $y$ being adjacent if there is a walk of length $k$ connecting $x$ and $y$ in $G$. 
This graph would be loopless only if $k$ is odd and $G$ has odd-girth at least $k+2$, thus this will be the only case of
interest for us in this work. If $\phi$ is a homomorphism of $G$ to $H$, then it can easily be checked that $\phi$ is also 
a homomorphism of $G^{(k)}$ to $H^{(k)}$. Thus to prove our claim we will prove the following stronger result.

\begin{theorem}\label{oddcase}
 There is a planar graph $G$ of odd-girth $2k+1$ with $\omega(G^{(2k-1)})\geq 2^{2k}$.
\end{theorem}

To prove the theorem we will in fact construct an example of such a graph. This construction is based on the following local
construction.

\begin{lemma}\label{SubdivisionOfK4}
 Let $G$ be a graph obtained from subdividing edges of $K_4$ such that in a planar 
 embedding of $G$ each of the four faces is a cycle of length $2k+1$. 
 Then $G^{(2k-1)}$ is isomorphic to $K_{4k}$.
\end{lemma}

\begin{proof}
Let $a,b,c$ and $d$ be the original vertices of the $K_4$ from which $G$ is constructed. 
For $x, y \in\{a, b, c, d\}$ let $P_{xy}$ be the subdivision of $xy$, and let $t_{xy}$ be the length of this path.
For an internal vertex $w$ of $P_{xy}$, let $P_{xw}$ (or $P_{wx}$) be the part of $P_{xy}$ connecting $w$ to $x$, let $t_{xw}$
be the length of it.

We have

\begin{align}\label{eqn face equal}\nonumber
t_{ab}+t_{bc}+t_{ca} &= t_{ab}+t_{bd}+t_{da}\\ \nonumber
&= t_{ac}+t_{cd}+t_{da}\\ \nonumber
&= t_{bc}+t_{cd} + t_{db}\\  
&= 2k+1.
\end{align}

From equation~\ref{eqn face equal} we have

\begin{align}\label{eqn equal parallel}
t_{xy} = t_{wz} \text{ for  }  \{x,y,w,z\} = \{a,b,c,d\}, 
\end{align}
that is to say the if all four faces are of a same length, then parallel edge of $K_4$ are subdivided the same number of times 
(the parity of the length of faces is not important for this claim, the even case will be used later).

Let $u$ and $v$ be a pair of vertices
of $G$. If they are both vertices of a facial cycle of $G$, then 
there is a walk of length  $2k-1$ connecting them since each facial cycle is of length $2k+1$. 
If there is no facial cycle of $G$ containing both $u$ and $v$, then they are internal vertices (after subdivision) of two distinct
parallel edges of $K_4$, thus we may assume, without loss of generality, that $u$ is 
 a vertex of the path $P_{ab}$ and $v$ is a vertex of the path $P_{cd}$. 
 

  Note that by equation~\ref{eqn equal parallel} we have
  
  \begin{align}\label{eqn break-up} \nonumber
  t_{au} + t_{bu} &= t_{cv} + t_{dv} \\
  &= t_{ab} = t_{cd}. 
  \end{align}

 If $t_{ab}=t_{cd}$ is even (odd respectively), 
 then  $t_{au}$ and  $t_{bu}$ have the same parity (different parities respectively) and 
 $t_{cv}$ and  $t_{dv}$ have the same parity (different parities respectively). 
 Moreover, since $t_{cd}$ is even (odd respectively) and $t_{ac}+t_{cd}+t_{da}=2k+1$, 
 $t_{ac}$ and $t_{ad}$ have different parities (same parity respectively). 
 

 Now one of the paths connecting $u,v$, say $P_{ua}\cup P_{ac}\cup P_{cv}$, is of length $t_{au} + t_{ac} + t_{cv}$, 
 and another path, say $P_{ub}\cup P_{bd}\cup P_{dv}$, is of length $t_{bu} + t_{bd} + t_{dv}$. 
 By (\ref{eqn break-up}) we have $(t_{bu} + t_{bd} + t_{dv})+(t_{au} + t_{ac} + t_{cv})=2(t_{ab}+t_{bd})$, hence 
 $t_{bu} + t_{bd} + t_{dv}$ and $t_{au} + t_{ac} + t_{cv}$ have a same parity. Furthermore, since $P_{ab}\cup P_{ad}\cup P_{bd}$ 
 forms a facial cycle we have $t_{ab}+ t_{ad}+t_{bd}=2k+1$, thus $2(t_{ab}+t_{bd})=4k+2-2t_{bc}\leq 4k$.

%

Hence we have $min\{(t_{au} + t_{ac} + t_{cv}), (t_{bu} + t_{bd} + t_{dv})\} \leq 2k$. 
Similarly, we can show that $min\{(t_{au} + t_{ad} + t_{dv}), (t_{bu} + t_{bc} + t_{cv})\} \leq 2k$.

\medskip

But note that $min\{(t_{au} + t_{ac} + t_{cv}), (t_{bu} + t_{bd} + t_{dv})\} $ and  
$min\{(t_{au} + t_{ad} + t_{dv}), (t_{bu} + t_{bc} + t_{cv})\}$ have different parities irrespective of the parity of $t_{ab}=t_{cd}$. Therefore, there is a walk of length $2k-1$ from $u$ to $v$. 
\end{proof}

 \medskip

\begin{proof}[\textbf{Proof of Theorem~\ref{oddcase}}]
 Consider a $K_4$ on four vertices $a,b,c$ and $d$. Let $G_1$ be a subdivision of this $K_4$ where edges $ab$ and $cd$
 each are subdivided into $2k-1$ edges. Thus $G_1$ is a subdivision of $K_4$ in which all the four faces are cycles of length $2k+1$.
 Hence by Lemma~\ref{SubdivisionOfK4} we have
 
 \begin{align}\nonumber
 \omega (G_1^{(2k-1)})=|V(G_1)|=4k.
 \end{align} 
 
 In the following we build a sequence of graphs $G_i$,$i=1,2,\cdots, 2k-1$,
 such that each $G_{i+1}$, $i\leq 2k-2$, contains $G_i$ as a subgraph, $G_{i+1}$ is planar 
 and of odd-girth $2k+1$ and such that
 $\omega(G_{i+1}^{(2k-1)}) > \omega(G_{i}^{(2k-1)})$. At the final step we will have
 
 \begin{align}\nonumber
 \omega(G_{2k-1}^{(2k-1)})\geq 2^{2k}.
 \end{align}
 
 We start with the following partial construction. Suppose $G_i$ is built and 
 let $P=uv_1v_2\cdots v_rw$ 
 be a maximal thread, that is, a path $P$
 connecting $u$ and $w$ such that all $v_j$'s $(j \in \{1,2,...,r\})$ are of degree 2 in $G_i$ but $u$ and $w$ 
 are of degree at least 3. 
 Furthermore, assume
 that $P$ is either  part of a path of length $2k-1$ connecting $a$ and $b$  or
 part of a path of length $2k-1$ connecting $c$ and $d$.

 Since $P$ is a thread, if we add a new edge $uw$ in  $G_i$, the resulting graph will still be planar. 
 So we add such an edge and
 subdivide it $r$ times to obtain the new thread $P'= uv_1'v_2'\cdots v_r'w$.
 Consider a planar drawing of the graph in which $P$ and $P'$ form a facial cycle of length $2r+2$.
 In the face $PP'$ connect $v_1$ and $v_r'$ by a new edge. 
 Subdivide this new edge $2k-r-1$ times (that is, into $2k-r$ edges, we draw it in dot line), 
 so that each of the facial cycles containing the new thread is of length $2k+1$.

 \begin{figure}[ht]
\begin{center}
\includegraphics[height=8cm]{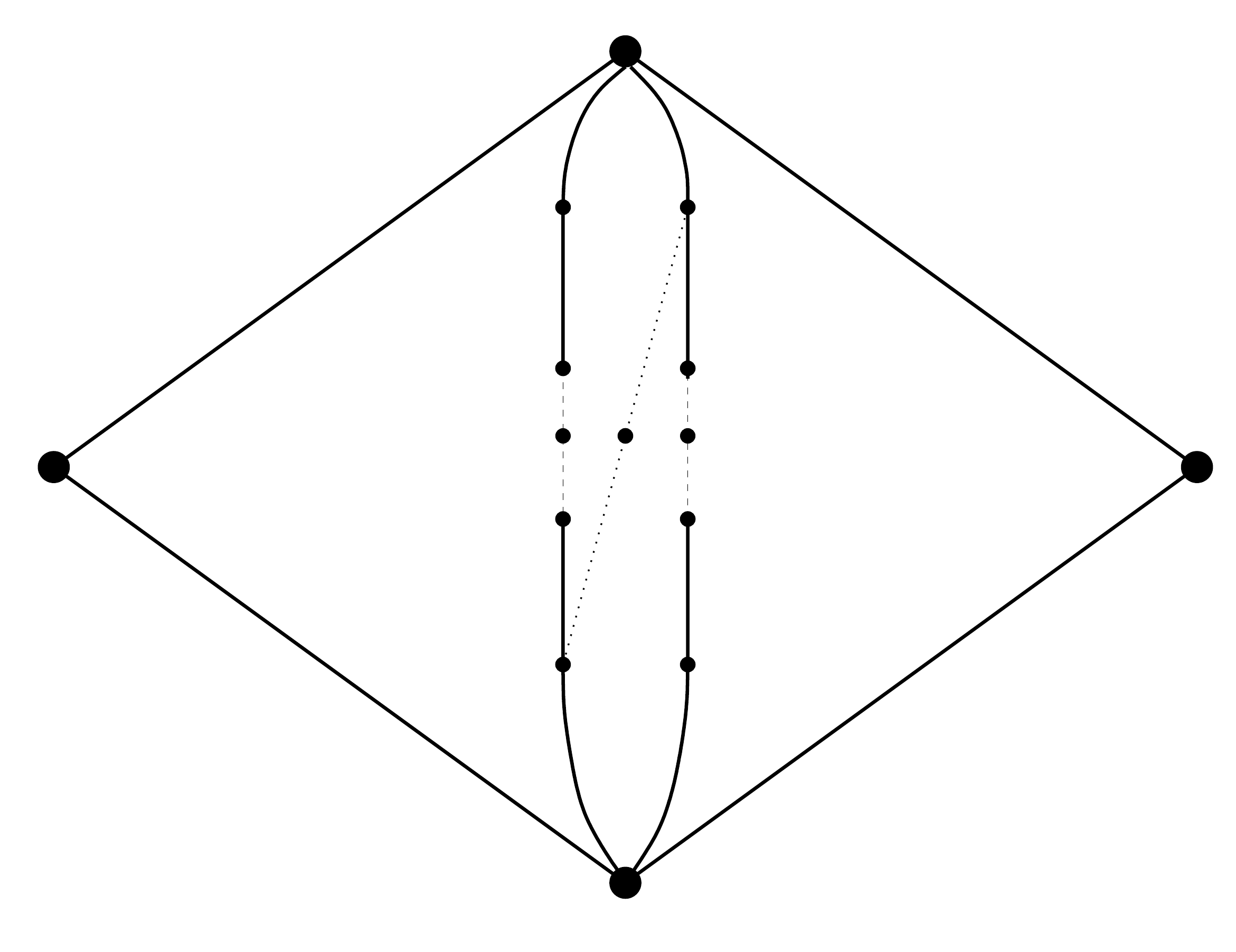}\end{center}
\caption{thread P}\label{threadP}
\end{figure}

 Denote by $G_i'$ the resulting graph.
 We first note that $G_i'$ is also of odd-girth $2k+1$. 
 Now suppose a maximal clique $W$ of $G_i^{(2k-1)}$ contains $v_j$ of the thread $P$. Then we claim that $W \cup {v_j'}$ is also a
 clique of $G_i'^{(2k-1)}$. 
 
 To prove this let $x$ be any vertex of $W$. If $x$ is not in $P$, then consider a  walk of length 
 $2k-1$ from $v_j$ to $x$.
 Each time this walk uses a part of $P$, replace it with the corresponding part from $P'$ and this would give a
 walk of length $2k-1$ connecting $x$ to $v_j'$.

 If $x \in P$, then, without loss of generality, 
 assume
 that $P$ is  part of a path of length $2k-1$ connecting $a$ and $b$.
  Consider the subgraph induced by this path together with $c$, $P'$ and the
 $v_1...v_r'$ thread we added to build $G'_i$.
 This induced subgraph is  a subdivision of $K_4$ in which all the faces are cycles of length $2k+1$. 
 Thus, by Lemma~\ref{SubdivisionOfK4} there is a walk of length $2k-1$ connecting $x$ and $v_j'$. 
 Extending this argument we observe that if all vertices of $P$ are in $W$, then $W\cup \{v_1',v_2', \cdots, v_r'\}$ 
 is a clique in $G_i'^{(2k-1)}$.

 Now we describe our general construction. At first we have $G_1$ on $4k$ vertices and two maximal threads.
 By Lemma~\ref{SubdivisionOfK4} all the vertices of these two threads are parts  of the unique clique
 of order $4k$ in $G_1^{(2k-1)}$.
 We apply the previously mentioned construction on both threads to build $G_2$ which will have four maximal threads each of 
 length $2k-1$ (we are only considering maximal threads that are part of a path of length $2k-1$ connecting $a,b$ or $c,d$, 
 for example the $v_{2k-2}'v_1$-thread drawn in dot line of Figure \ref{threadP} is not considered). 
 There is a clique of order $4k+2(2k-2)$ in $G_2^{(2k-1)}$, and there are four maximal threads of length $2k-2$, 
 each is a part of a path  of length $2k-1$ either connecting $a$ and $b$ or  $c$  and $d$. 
 
  Continuing this construction, in general, there is a clique $W_i$ of $G_i^{(2k-1)}$ ($2\leq i\leq 2k-1$) which is of order
 $4k+\sum_{j=1}^{i-1} 2^j(2k-j-1)$ and there are $2^{i}$ maximal threads of length $2k-i$ which each is a part of a  path of
 length $2k-1$ connecting $a$ to $b$ or $c$ to $d$. 
 
 Note that $G_i$ at each step is a planar graph of odd-girth $2k+1$. The clique $W_{2k-1}$ of $G_{2k-1}^{(2k-1)}$
 has order equal to
 
 \begin{align}\nonumber
 4k+\sum_{j=1}^{2k-2} 2^j(2k-j-1) &=  4k + (2k-1) \sum_{j=1}^{2k-2} 2^j - 2\sum_{j=1}^{2k-2} j2^{j-1}\\ \nonumber
 &= 4k + [(2k-1)(2^{2k-1}-2)] - \\ \nonumber
 &\hspace{.5cm}2[(1 - 2^{2k-1}) - (-1)(2k-1)2^{2k-2}] \\ \nonumber
 &= 4k + (k2^{2k} - 4k - 2^{2k-1} + 2) - \\ \nonumber
 &\hspace{.5cm} (2 - 2^{2k} + k2^{2k} - 2^{2k-1}) \\ \nonumber
 &=2^{2k}.\nonumber
 \end{align}  

 This completes the proof. 
\end{proof}

\begin{corollary}
 Let $B$ be a graph of odd-girth $2k+1$ to which every planar graph of odd-girth 
 $2k+1$ admits a homomorphism. 
 Then $|V(B)|\geq 2^{2k}$. 
 Furthermore, if $B$ is minimal with this property, then $\delta(B)\geq 2k+1$.
\end{corollary}

\begin{proof}
 Let $G$ be a graph build in the previous theorem. Since $G$ is of odd-girth $2k+1$, 
 by the assumption, it maps to $B$.  Since $B$ is also of odd-girth $2k+1$, 
 both $B^{(2k-1)}$ and $G^{(2k-1)}$ are simple graphs and $G^{(2k-1)}\to B^{(2k-1)}$. 
 Hence $K_{2^{2k}}\subset B^{(2k-1)}$ which, in particular implies $|V(B)|\geq 2^{2k}$.

 To prove the lower bound on minimum degree, we first introduce the following graph:
 let $P$ be a path of length $2k$ on vertices $x_1, x_2, \cdots, x_{2k+1}$ connected 
 in this order. Let $P'$ be obtained from $P$ by subdividing each edge $2k-2$ 
 times so that $x_i$ is at distance $2k-1$ from $x_{i+1}$. 
 Let $y_1^i,  y_2^i, \cdots y_{2k-2}^i$ be the new vertices subdividing 
 $x_ix_{i+1}$ and connected in this order in $P'$.  We add short
 cut edges so that the shortest odd walk between each $x_i$ and $x_j$ becomes
 of length $2k-1$. These edges are $x_1y_1^2$, $y_1^2y_2^3$, $y_2^3y_3^4, \cdots, 
 y_{2k-2}^{2k}x_{2k+1}$. Now given a vertex $u$, the graph $P_u$ is a graph which is formed 
 from a disjoint copy of $P'$ by connecting $u$ to all $x_i$'s. 
 Note that the graph $P_u$ is of odd-girth $2k+1$ and that in $P_u^{(2k-1)}$ vertices of 
 $P$ (i.e., $x_i$'s) induce a $(2k+1)$-clique. 
 
 Now since $B$ is minimal, there exists a planar graph $G_B$ of odd-girth
 $2k+1$ whose mappings to $B$ are always onto. Let $G_B^{*}$ be a new graph obtained 
 from $G_B$ by adding a $P_u$ for each vertex $u$ of $G_B$. This new graph is also of 
 odd-girth $2k+1$, thus, by the choice of $B$, it maps to $B$. Let $\phi$ be such a 
 mapping of $G_B^{*}$ to $B$. This mapping induces a mapping of $G_B$ to $B$. Thus
 each vertex $v$ of $B$ is image of a vertex $u$ of $G_B$ by the choice of $G_B$. 
 But in the mapping $G_B^{*}$ to $B$, all $x_i$'s of $P_u$ must map to distinct
 vertices all of whom are neighbours of $\phi(u)=v$.
 
\end{proof}

Note that since $PC(2k)$ is a $(2k+1)$-regular graph on $2^{2k}$ vertices, it would be an optimal bound if 
Conjecture~\ref{PlanarsToProjectiveCubes} holds.

\section{Optimal bound for planar signed bipartite graphs}


The development of the notion of homomorphisms for signed graphs has began very recently and, therefore, it is not yet known
if an analogue of Theorem~\ref{NOduality} would hold for the class of signed bipartite graphs. While we believe that would be the 
case, here we prove that $SPC_k$ is the optimal bound for the signed bipartite case of Conjecture~\ref{PlanarsToProjectiveCubes}
if the conjecture holds.

To start, we introduce an analogue notion of walk-power. Let $[G, \Sigma]$ be a signed bipartite graph with $(X,Y)$ being
the partition of vertices. Given an even integer $r\geq 2$ we define $[G, \Sigma]^r$ to be a graph on $V(G)$ where a pair $u, v$
of vertices are adjacent if the following conditions hold: 
\begin{itemize}
 \item $u$ and $v$ are in the same part of $G$;
 \item there are $u,v$-paths $P_1$ and $P_2$, each of length at most $r$, 
 such that one has an odd number of negative edges and the other has even 
 number of them.
\end{itemize}
  Note that the second condition is independent of the choice of a representative
  signature. Furthermore $[G, \Sigma]^r$ is a graph (not signed) with no connection 
  from $X$ to $Y$.

We remark that these two conditions together are to say that: for any choice of an equivalent signature if $u$ and $v$ 
are identified then there would be an unbalanced cycle of even length at most $r$. That can be analogue to the definition
of $G^r$ for odd values of $r$ where odd-girth of $G$ is at least $r+2$, in the following sense: first of all $G$ can be 
regarded as a consistent signed graph $[G, E(G)]$; secondly for any choice of equivalent signature $\Sigma$ of $[G, E(G)]^r$
if identifying pair $u, v$ of vertices results in yet an odd (signed) graph (analogue of 1) but of unbalanced girth at most \
$r$ (analogue of 2), then $u$ and $v$
are adjacent in  $[G, E(G)]^r$. While $[G, E(G)]^r$ could be a proper subgraph of 
$G^r$, the claim and proof of Theorem~\ref{oddcase} can 
be revised with this modified definition.

With the previous remark following lemma is easy to verify.

\begin{lemma}\label{MappingPowerSignedBipartite}
 Let $[G, \Sigma]$ and $[H, \Pi]$ be two signed bipartite graphs and let $\phi$ be a homomorphism of $[G, \Sigma]$ to $[H, \Pi]$.
 Then for any positive even integer $r$, $\phi$ is also a homomorphism of $[G, \Sigma]^r$ to  $[H, \Pi]^r$.
\end{lemma}

Thus if both graphs are of unbalanced girth at least $r+2$, then $[G, \Sigma]^r$ and $[H, \Pi]^r$ are both loopless, and, therefore,
mapping $\phi$ would imply $\omega ([G, \Sigma]^r) \leq \omega ([H, \Pi]^r)$. Furthermore, assuming that $G$ and $H$ are both
connected, since $\phi$ is also a mapping of $G$ to $H$, it would preserve bipartition. Thus in what follows we will built a signed
bipartite planar graph $[G,\Sigma]$ of unbalanced girth $2k$ such that each part of $G$ contains a clique of size 
$2^{k-2}$ in $[G, \Sigma]^{2k-2}$.

To this end we start with the following lemma which is the signed bipartite analogue of Lemma~\ref{SubdivisionOfK4}.

\begin{lemma} \label{signedSubdivisionOfK4}
 Let $[G,\Sigma]$ be a planar signed graph which is obtained from assigning a signature to a subdivision of $K_4$ in such a way 
 that each of the four facial cycles is an unbalanced cycle of length $2k$. 
 Then $[G,\Sigma]^{(2k-2)}$ is isomorphic to disjoint copies of $K_{(2k-1)}$ each induced on one part of the bipartite graph $G$. 
\end{lemma}

\begin{proof}
We consider a fixed signature $\Sigma$ of  $[G,\Sigma]$. We will use the same notations ($P_{xy}$, $t_{xy}$ etc.) as in
Lemma~\ref{SubdivisionOfK4}. 
Thus as proved in that lemma, parallel edges of $K_4$ are subdivided same number of times. 
Furthermore, repeating the same argument modulo 2, we can conclude that 
the number of negative edges in $P_{xy}$  and the number of negative edges 
in $P_{wz}$ have same parity   for all $\{x,y,w,z\} = \{a,b,c,d\}$.

Let $u$ and $v$ be two vertices from same part of $G$ (thus any  path connecting $u$ and $v$ have even length). 
We would like to prove that they are adjacent in $[G,\Sigma]^{(2k-2)}$.
If they belong to a same facial cycle, then the two paths connecting these two vertices in that (unbalanced) cycle 
satisfy the conditions and we are done. 
Hence, assume without loss of generality that $u \in P_{ab}$ and $v \in P_{cd}$.

Removing the edges of the parallel paths $P_{ad}$ and $P_{bc}$ will result in a cycle of length $4k-2t_{ad}$ containing $u,v$.
This implies:

\begin{align}\label{u-vPath1}
(t_{ua} + t_{ac} + t_{cv}) + (t_{ub} + t_{bd} + t_{dv}) \leq 4k-2 \nonumber \\
\Rightarrow \min \{(t_{ua} + t_{ac} + t_{cv}), (t_{ub} + t_{bd} + t_{dv})\} \leq 2k-2.
\end{align}

Similarly by removing $P_{ac}$ and $P_{bd}$ we get

\begin{align}\label{u-vPath2}
\min\{(t_{ua} + t_{ad} + t_{dv}), (t_{ub} + t_{bc} + t_{cv})\} \leq 2k-2.
\end{align}

It remains to show that the two paths of equation~(\ref{u-vPath1}) and (\ref{u-vPath2}) have different number of negative edges modulo 2. 
To see this note that union of any one of the two paths from~(\ref{u-vPath1}) with a path from~(\ref{u-vPath2}) 
covers a facial 
cycle exactly once and one a part of $P_{ab}$ or $P_{cd}$ twice. 
Since each facial cycle is unbalanced, our claim is proved.
\end{proof}

We are now ready to present our general construction.

\begin{theorem}
 There exists a planar signed bipartite graph $[G,\Sigma]$ such that in $[G,\Sigma]^{(2k-1)}$ each part of $G$ induces a clique
 of size at least $2^{2k-2}$.
\end{theorem}

\begin{proof}
 Consider a $K_4$ on four vertices $a,b,c$ and $d$. Let $G_1$ be a subdivision of this $K_4$ where edges $ab$ and $cd$
 each are subdivided into $2k-2$ edges. Note that $G_1$ is a connected bipartite graph and let $V_1$ and $V'_1$ be its partite sets.
Let $\Sigma_1$ be the signature  with the new edge
  incident to $a$ (created by the subdivision of $ab$) and the new edge 
 incedent to $c$ (created by the subdivision of $cd$) being negative. 
 Thus the signed bipartite graph 
 $[G_1,\Sigma_1]$ is a subdivision of $K_4$ in which all the four faces are unbalanced cycles of length $2k$.
 Hence by Lemma~\ref{signedSubdivisionOfK4} we 
 know that each of $V_1$ and $V'_1$ induces a clique of order 
 $2k-1$ in 
 $[G_1,\Sigma_1]^{(2k-2)}$.

 In the following we will build a sequence of signed graphs $[G_i,\Sigma_i]$, for $i \in \{1,2,\cdots, 2k-1\}$,
 such that each $[G_{i+1},\Sigma_{i+1}]$, $i\leq 2k-2$, contains $[G_i,\Sigma_i]$ as a subgraph.
 Moreover, the signed graph $[G_{i},\Sigma_{i}]$ is a bipartite planar graph with unbalanced-girth $2k$ and partite sets $V_i$, $V'_i$.
Let us denote the clique number of the graph  induced by $V_i$ (or $V'_i$) from $[G_i,\Sigma_i]^{(2k-2)}$ by $f(i)$ (or $f'(i)$). 
Note that both the functions are strictly increasing and at the final step we will have 
$$f(2k-1),f'(2k-1) \geq 2^{2k-2}.$$

 We start with the following partial construction. Suppose $[G_i,\Sigma_i]$ is built and 
 let $P=uv_1v_2\cdots v_rw$ 
 be a maximal thread. 
 Furthermore, assume that $P$ is either  part of a path of length $2k-2$ connecting $a$ and $b$  or
 part of a path of length $2k-2$ connecting $c$ and $d$.

 Since $P$ is a thread, if we add a new edge $uw$ in  $[G_i,\Sigma_i]$, the resulting graph will still be planar. 
 So we add such an edge and subdivide it $r$ times to obtain the new thread $P'= uv_1'v_2'\cdots v_r'w$.
 Also we assign signs of the new edges in such a way that the edges  $uv_1'$ and $v_rw$ have a same sign,
 the edges $v_r'w$ and $uv_1$ have a same sign and the edges $v_i'v_{i+1}'$ and  $v_{r-i+1}v_{r-i}$ have a same sign.
 
%
 
 Consider a planar drawing of the graph in which $P$ and $P'$ form a facial cycle of length $2r$.
 In the face $PP'$ connect $v_1$ and $v_r'$ by a new edge. 
 Subdivide this new edge $2k-r-2$ times (that is, into $2k-r-1$ edges, we color them green), so that each of the facial cycles
 containing the new thread is of length $2k$. Choose signs of the edges of this new path in such a way that
 each of the facial cycles containing the new thread is unbalanced. 
 
 Let $[G_i', \Sigma_i']$ be the resulting signed graph.
 We first note that $[G_i',\Sigma_i']$ is also planar bipartite of unbalanced-girth $2k$.
Now suppose that the vertices of $P$, indexed by odd (or even) numbers  are all part of a maximal clique in the graph induced by one partite set of $G_i$ 
in  $[G_i,\Sigma_i]^{(2k-2)}$. Then the vertices of $P \cup P'$, 
indexed by odd (or even) numbers are all part of a maximal clique in the graph induced by the 
corresponding  partite set of $G'_i$ 
in  $[G'_i,\Sigma'_i]^{(2k-2)}$. 
This can be proved by similar logic used in the proof of Theorem~\ref{oddcase}.  
The only difference is that to prove the above claim one needs to repeat the argument
based on the parity of number of negative edges instead of the parity of number of edges and use Lemma~\ref{signedSubdivisionOfK4} instead of 
Lemma~\ref{SubdivisionOfK4}.

   Now we describe our general construction. 
   At first we have $[G_1,\Sigma_1]$ on $4k-2$ vertices and two maximal threads.
 By Lemma~\ref{signedSubdivisionOfK4} the vertices, indexed by  numbers with the same parity, 
 of these two threads
 are parts of the unique clique 
 of order $2k-1$ in their respective components in $[G_1,\Sigma_1]^{(2k-2)}$.
 We apply the previously mentioned construction on both the threads to build $[G_2,\Sigma_2]$ 
 which will have four maximal threads each of length $2k-3$ (we are only
 considering maximal threads that are part of a path of length $2k-2$ connecting $a,b$ or $c,d$, 
 the green threads are not considered). 
 There are two disjoint  cliques, each of order $(2k-1)+(2k-3)$ in $[G_2,\Sigma_2]^{(2k-2)}$, that is, $f(2)=f'(2) = (2k-1)+(2k-3)$, and there are four maximal threads of length $2k-3$, each a part of a path
 of length $2k-2$ either connecting $a,b$ or  connecting $c,d$. 
 

  Continuing this construction, in general, $f(i) = f'(i) =(2k-1)+\sum_{j=1}^{i-1} 2^{j-1}(2k-j-2)$ 
    and there are $2^{i}$ maximal threads of length $2k-i-1$ which are part of a  path of
 length $2k-2$ connecting $a,b$ or connecting $c,d$. 
 
 Note that $[G_i,\Sigma_i]$ at each step is a planar bipartite signed graph of unbalanced-girth $2k$. 
 Therefore
 
 \begin{align}\nonumber
 f(2k-1) = f'(2k-1) &= 2k-1 +\sum_{j=1}^{2k-2} 2^{j-1}(2k-j-2) =  2k-1 + (k-1) \sum_{j=1}^{2k-2} 2^j - \sum_{j=1}^{2k-2} j2^{j-1}\\ \nonumber
 &= 2k-1 + [(k-1)(2^{2k-1}-2)] - [(1 - 2^{2k-1}) - (-1)(2k-1)2^{2k-2}] \\ \nonumber
 &= 2k-1 + [k2^{2k-1} - 2k - 2^{2k-1} + 2] - [1 - 2^{2k-1} + k2^{2k-1} - 2^{2k-2}] \\ \nonumber
 &=2^{2k-2}.\nonumber 
 \end{align}  

 This completes the proof.
\end{proof}

\begin{corollary}
If $[B, \Pi]$ is a minimal signed bipartite graph of unbalanced-girth $2k$ 
to which every planar signed bipartite graph of unbalanced-girth $2k$ admits 
a homomorphism, then $|V(B)|\geq 2^{2k-1}$ and $\delta(B) \geq 2k$.
\end{corollary}

\begin{proof}
 Let $[G,\Sigma]$ be the graph built in the previous theorem. 
 Since $[G,\Sigma]$ is of unbalanced-girth $2k$, by the assumption, it maps to 
 $[B, \Pi]$.
 Since $[B, \Pi]$ is also of unbalanced-girth $2k$, both $[B, \Pi]^{(2k-2)}$ and 
 $[G,\Sigma]^{(2k-2)}$ are simple bipartite graphs 
 and $[G,\Sigma]^{(2k-2)}\to [B, \Pi]^{(2k-2)}$. 
 Hence there is a $K_{2^{2k-2}}$ in each part of $[B,\Pi]^{(2k-2)}$ which, 
 in particular implies $|V(B)|\geq 2^{2k-1}$. 
 
 To prove the lower bound on minimum degree, note that since $[B,\Pi]$ is minimal, 
 there exists  a signed bipartite planar graph $[G_B,\Sigma_B]$ of unbalanced-girth 
 $2k$ whose mappings to $[B,\Pi]$ are always onto. 
 Now consider the graph $[G,\Sigma]$ built in the previous theorem. 
  Note that in
 $[G,\Sigma]^{(2k-2)}$  all the $2k$ neighbours of $a$ are adjacent to each other 
 where $a$ is one of the vertices of the $K_4$ that we started our construction with.
  Therefore, in any mapping of $[G,\Sigma]$ to $[B,\Pi]$ image of $a$ must be of 
  degree at least $2k$.

 Now for each vertex $x$ of $[G_B,\Sigma_B]$ add a vertex disjoint copy $[G_x,\Sigma]$
 of $[G,\Sigma]$ and identify the vertex $a$ of
  $[G_x,\Sigma]$ with $x$. 
 Let $[G'_B,\Lambda_B]$ be the new graph. 
 By the construction, $[G'_B,\Lambda_B]$ is also a signed bipartite planar graph of unbalanced-girth $2k$. 
 Hence it maps to $[B,\Pi]$. 
 In any such mapping, by the choice of $[G_B,\Sigma_B]$ and construction of $[G'_B,\Lambda_B]$, each vertex of $[B,\Pi]$ is an image 
 of $a$ in a mapping
 of $[G,\Sigma]$ to $[B,\Pi]$. Thus each vertex of $[B,\Pi]$ has degree at least $2k$.
\end{proof}

\section{Concluding remarks}

P. Seymoure has conjectured in \cite{S75} that the edge-chromatic number of a planar multi-graph is equal to 
it fraction edge-chromatic number. 
It turns out that the restriction of this conjecture for $k$-regular multigraph can be 
proved if and only if Conjecture~\ref{PlanarsToProjectiveCubes} is proved for this 
value of $k=d-1$. 
This special case of Seymour conjecture is proved for $k\leq 8$ in series of work 
using induction and the Four-Color Theorem in \cite{G12} ($k=4,5$), \cite{DKK10} 
($k=6$), \cite{E11} ($k=7$) and \cite{CES12} ($k=8$). Thus 
Conjecture~\ref{PlanarsToProjectiveCubes} is verified for $d\leq 7$. 
Hence we have the following corollary. 

\begin{theorem}
 For $d\leq 7$ the signed graph $SPC_{d}$ is the smallest consistent graph 
 (both in terms of number of vertices and edges) of unbalanced-girth $d+1$
 which bounds all consistent planar signed graphs of unbalanced-girth at
 least $d+1$.  
\end{theorem}

B. Guenin has proposed a strengthening of Conjecture~\ref{PlanarsToProjectiveCubes}
by replacing the condition of planarity with no $(K_5, E(K_5))$-minor.

For further generalization one can consider the following general question:

\begin{problem}
Given $d$ and $r$, $d\geq r$ and $d=r \pmod 2$ what is the optimal bound of unbalanced
girth $r$ which bounds all consistent signed graph of unbalanced-girth $d$ with no  
$(K_n,  E(K_n))$-minor?
\end{problem}

We do not yet know of existence of such a bound in general. 
For $n=3$, consistent signed graphs with no $(K_n,  E(K_n))$-minor are
bipartite graphs with all edges positive, and, therefore, bounded by $K_2$.
For $n=5$ if the input and target graph are both of unbalanced-girth $d+1$, 
then our work and Geunin's extension of Conjecture~\ref{PlanarsToProjectiveCubes}
proposes projective cubes as the optimal solutions. For $d=r=3$, the answer would be
$K_{n-1}$ if Odd Hadwiger conjecture is true. For the case of $n=4$ some partial
answers are given by F. Foucaud and first author. For all other cases there is not
even a conjecture yet.

\bibliographystyle{alpha}
\newcommand{\etalchar}[1]{$^{#1}$}

\end{document}